\documentclass[12pt]{amsart}

\pdfoutput=1

\usepackage{latexsym}
\usepackage[centertags]{amsmath}
\usepackage{amsfonts}
\usepackage{amssymb}
\usepackage{amsthm}
\usepackage{newlfont}
\usepackage{graphics}
\usepackage{color}

\usepackage[usenames,dvipsnames]{xcolor}
\def\dessin#1#2{\includegraphics[#1]{#2}}

\usepackage[demo]{graphicx} 
\usepackage{wrapfig,floatrow}
\usepackage{lipsum} 
\usepackage[font=small,labelfont=bf,margin=5mm]{caption}

\RequirePackage{xparse, graphicx, caption, picins}
\DeclareDocumentCommand \addpic{O{0.4\textwidth} m g}{\parpic[r]{%
\begin{minipage}{#1}
    \includegraphics[width=\textwidth]{#2}%
    \IfNoValueTF{#3}{}{\captionof{figure}{\footnotesize #3}}
\end{minipage}
}}

\newtheorem{theo}{Theorem}

\newtheorem{lem} [theo]{Lemma}
\newtheorem{cor}[theo]{Corollary}

\newtheorem{rem}[theo]{Remark}

\def\orr{\overline{r}}
\def\oR{\overline{R}}

\newcommand{\area}{\operatorname{area}}

\begin{document}

.
\vskip -.5in

\title{Inverting the Rational Sweep Map}

\author{Adriano M. Garsia$^1$ and Guoce Xin$^2$}

\address{ $^1$Department of Mathematics, UCSD \\
$^2$School of Mathematical Sciences, Capital Normal University,
Beijing 100048, PR China}

\email{$^1$\texttt{garsiaadriano@gmail.com}\ \  \&\small $^2$\texttt{guoce.xin@gmail.com}}

\date{March 20, 2016} 

\begin{abstract}
We present a simple algorithm for inverting the sweep map on rational $(m,n)$-Dyck paths for a  co-prime pair $(m,n)$ of positive integers. This work is inspired by Thomas-Williams work on the modular sweep map.
A simple proof of the validity of our algorithm is included.
\end{abstract}

\maketitle

\section{The Algorithm}
Inspired by the Thomas-William algorithm \cite{Nathan}  for inverting the general modular sweep map,
we find a simple algorithm to invert the sweep map for rational Dyck paths.
The fundamental fact that made it so difficult to invert the sweep map in this case is that all  previous attempts used only   the ranks
of the vertices of the rational Dyck paths. Moreover the geometry of rational Dyck paths   was not consistent with those  ranks.

A single picture will be sufficient here to understand the idea.
In what follows, we always denote by $(m,n)$ a co-prime pair of positive integers, South end  (by letter $S$) for the starting point of a North step and West end (by letter $W$) for the starting point of an East step, unless specified otherwise. This is convenient and causes no confusion because we usually talk about the starting points of these steps.

\begin{figure}[!ht]
\centering{
\mbox{\dessin{width=3.5 in}{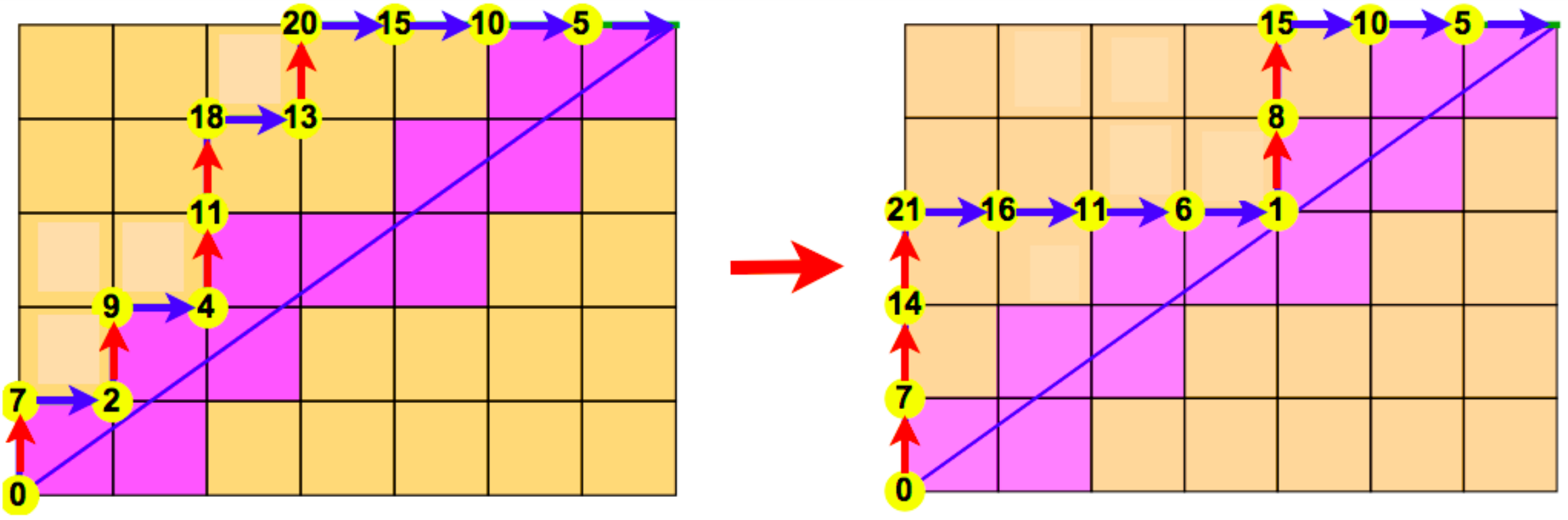}}
$$\bar D\hspace{2cm} \rightarrow \hspace{2cm} D$$
\caption{A rational $(7,5)$-Dyck path and its sweep map image.}
\label{fig:SweepMap}
}
\end{figure}

Figure \ref{fig:SweepMap} illustrates a rational $(m,n)$-Dyck path $\overline D$ for $(m,n)=(7,5)$
and its sweep map image $D$ on its right.
Recall that the ranks of the starting vertices of an $(m,n)$-Dyck path $\overline D$ are recursively  computed   starting with rank $0$, and adding $m$ after a North step and subtracting $n$ after an East step as shown in Figure \ref{fig:SweepMap}.

To obtain the Sweep image $D$ of $\overline D$, we let the main diagonal (with slope $n/m$) sweep from right to left and successively draw the steps of $D$ as follows: i) draw a South end (and hence a  North step) when we sweep a South end of $\overline D $; ii) draw a West end (hence an  East step) when we sweep a West end of $\overline D$. The steps of $D$ can also be obtained by rearranging the steps of  $\overline D $ by increasing ranks of their starting  vertices. The  sweep map has become an active subject in the recent 15 years. Variations and extensions have been found, and some classical bijections turn out to be
the disguised version of the sweep map. See \cite{sweepmap} for detailed information and references.

The open problem was the reconstruction of the path on the left from the path on the right.
The idea that leads to the solution of this problem is to draw these two paths as  in Figure \ref{fig:SweepMapNew}.

\begin{figure}[!ht]
\begin{center}
 \mbox{\dessin{width=4 in}{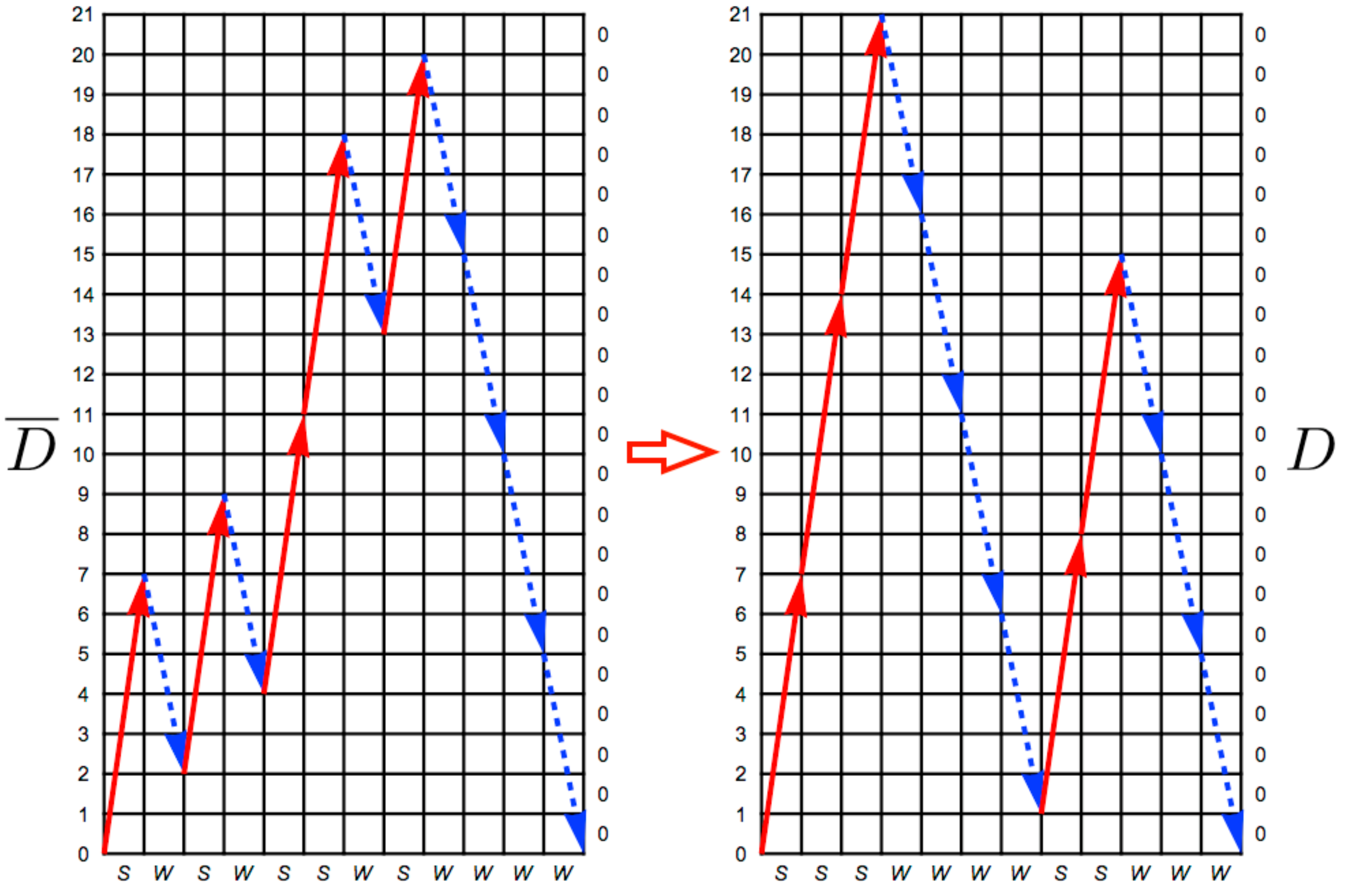}}
\caption{Transformation of the $(7,5)$-Dyck paths in Figure \ref{fig:SweepMap}.}
\label{fig:SweepMapNew}
\end{center}
\end{figure}

That is we first  stretch all the arrows so that their lengths correspond to the effect they have on the ranks of the vertices of the path then add an appropriate
clockwise rotation  to obtain the two path diagrams in Figure \ref{fig:SweepMapNew}.
The path diagrams are  completed by writing an $S$ for each South end in our original path and a $W$ for each West end.
On the left we have added a list of each level. The ranks of $\overline D$ become visually the levels of the staring points of the arrows.
On the right, at each level we count the red (solid) segments and the blue (dashed)\footnote{Suggested by the referee, we have drawn blue dashed arrows for convenience of black-white print. We will only use ``red" and ``blue" in our transformed Dyck paths, but in our context, red, solid, up and positive slop are equivalent; blue, dashed, down and negative slop are equivalent.} segments which traverse that level and record their difference. Of course these differences (called row counts) turn out to be all equal to $0$, for obvious reasons. This will be referred to as the $0$-row-count property. Theorem \ref{t-characterization} states that this is a characteristic property of rational Dyck paths, which becomes evident when paths are drawn in this manner. This fact is conducive to the discovery of  our algorithm for  constructing the pre-image of any $(m,n)$-Dyck path.

\begin{figure}[!ht]
\centering{
\mbox{\dessin{width=3.8 in}{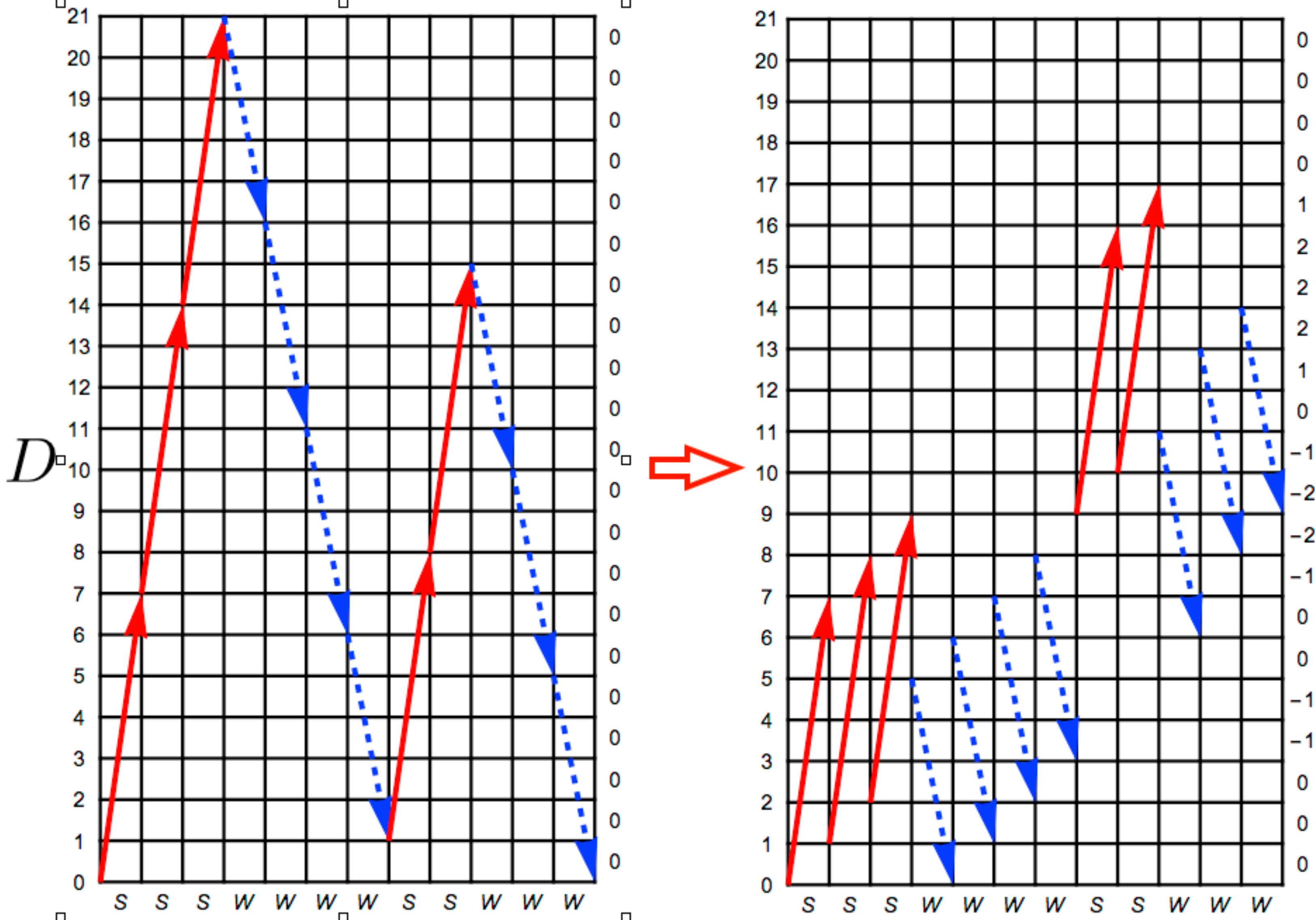}}
\caption{A given rational Dyck path and its starting path diagram on the right.}
\label{fig:InitialRank}
}
\end{figure}

The first step in our algorithm is to reorder  the
arrows of the path on the left of Figure  3, 
so that the ranks of their starting points are
minimally strictly increasing. More precisely the first three red arrows are
lowered in their columns to start at levels $0,1,2$.  To avoid placing part of the first blue arrow below level $0$ we lower it to start at level $5$. This done all the remaining arrows are successively placed to start at levels $6,7,8,9,10,11,12,13$. Notice the row counts at the right of the resulting path diagram. Our aim is to progressively reduce them  all to zeros, which are the row counts
characterization of the path diagram we are working to reconstruct.


\begin{figure}[!ht]
\centering{
\mbox{\dessin{height=5 in}{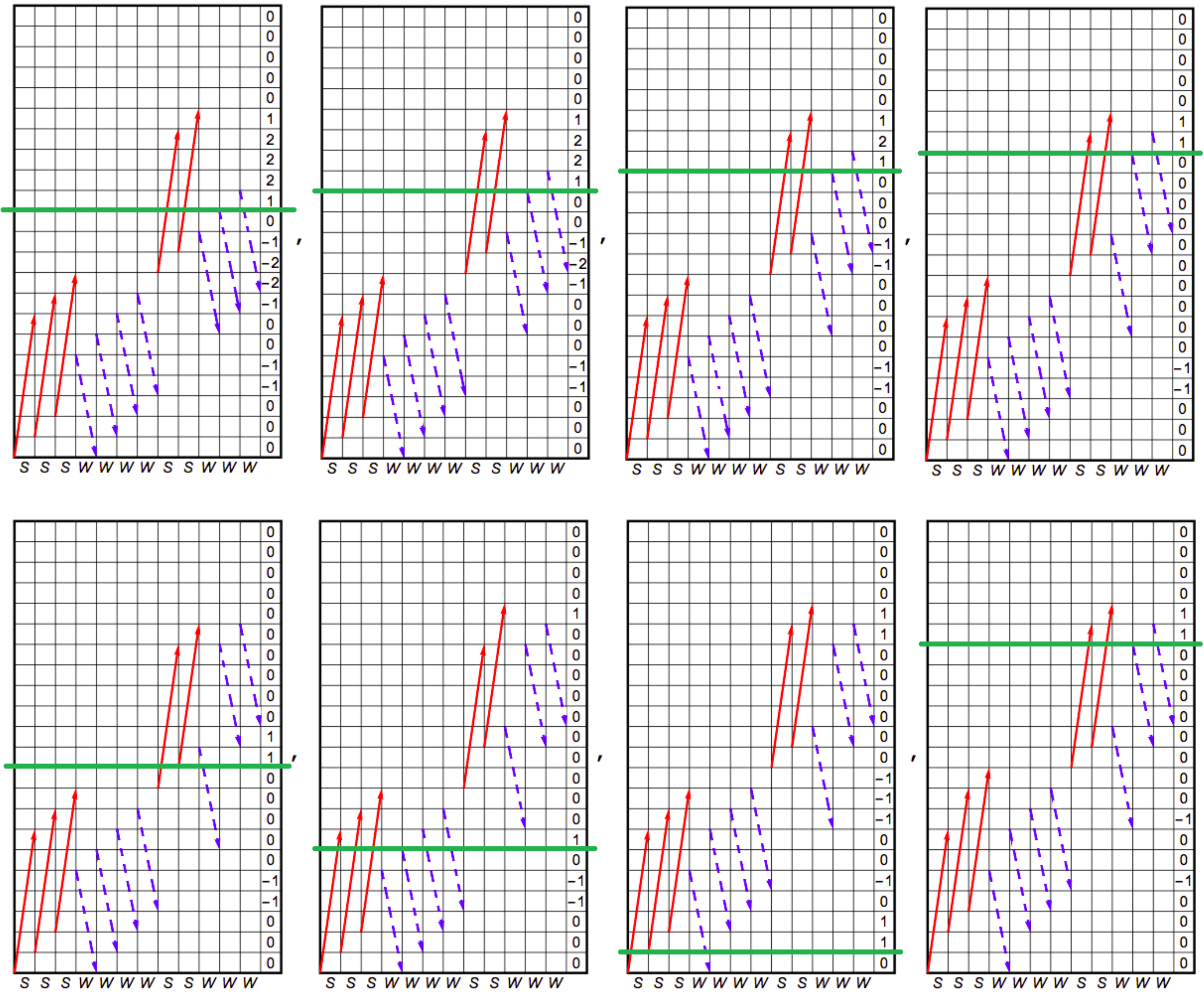}}\\[2ex]
\mbox{\includegraphics[height=2.43 in]{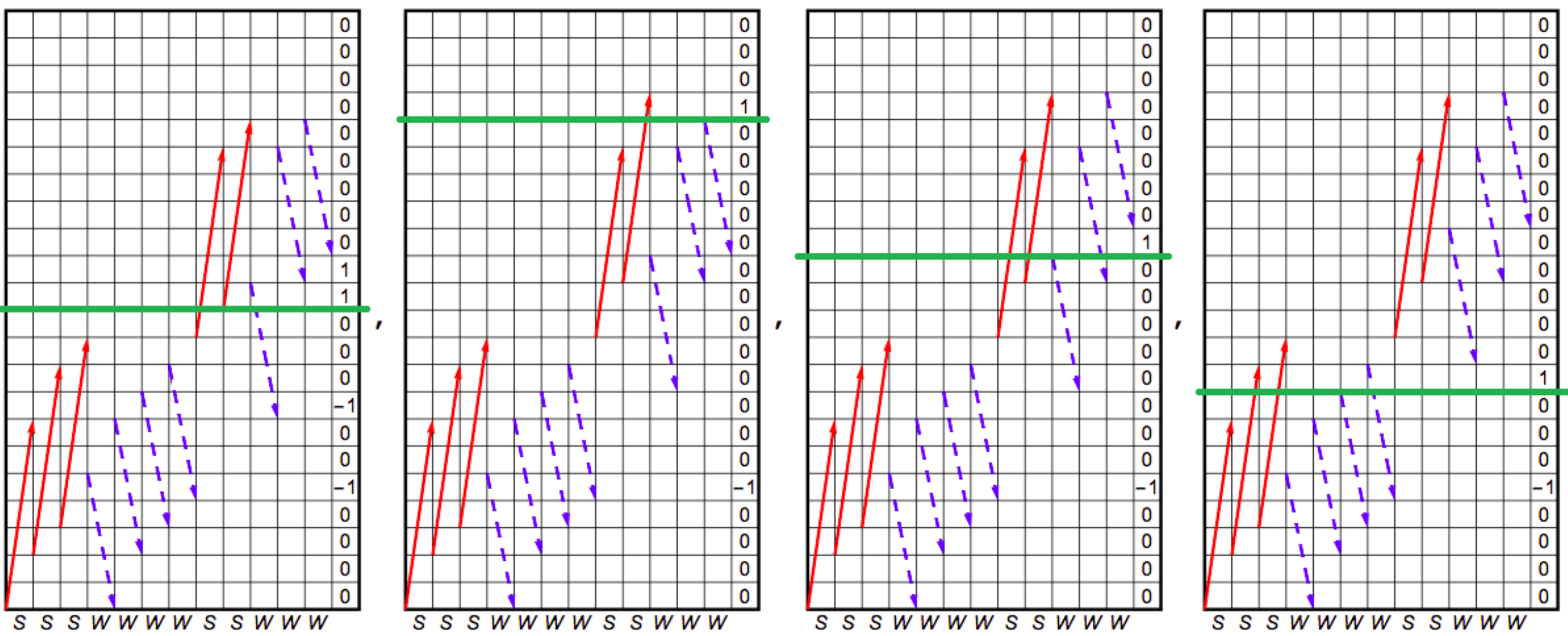}}
\caption{Part 1 of the 18 path diagrams that our algorithm produced.}
\label{fig:AlgorithmResult1}
}
\end{figure}
\begin{figure}[!ht]
\centering{
\mbox{\includegraphics[height=5 in]{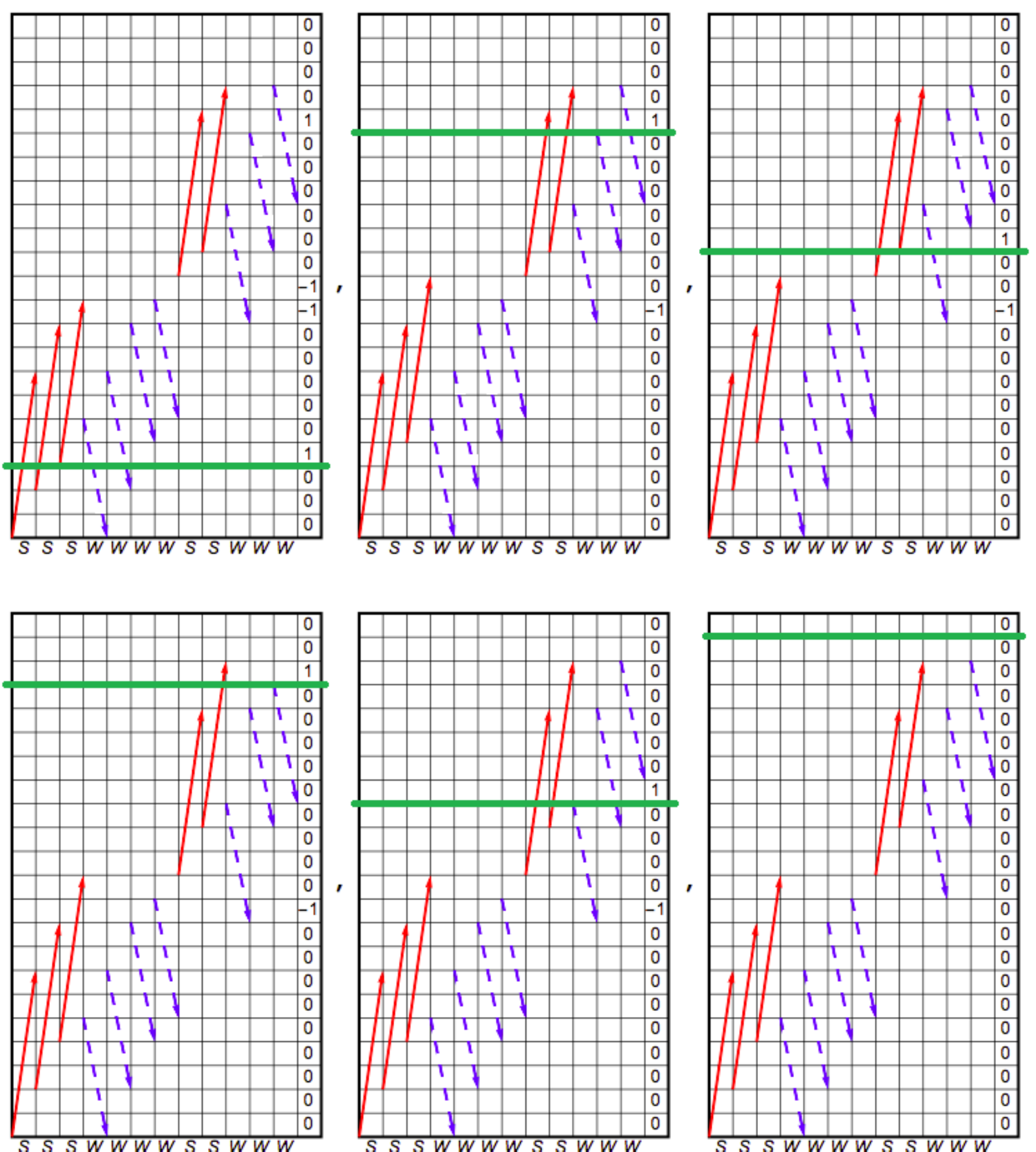}}
\caption{Part 2 of the 18 path diagrams that our algorithm produced.}
\label{fig:AlgorithmResult2}
}
\end{figure}

The miracle is that this can be achieved  by a sequence of identical steps.
 More precisely, at each step of our algorithm we locate the lowest row
 sum that is greater than $0$. We  next notice that there is a unique arrow that  starts immediately below that   row
 sum.  This done we move that arrow
 one unit upwards. However,  to keep the ranks strictly increasing we also shift, when necessary,  some  of the successive arrows by one unit upwards.
In this particular case our MATHEMATICA implementation of the resulting algorithm produced the sequence of 18 path diagrams in Figures \ref{fig:AlgorithmResult1} and
\ref{fig:AlgorithmResult2}. Notice, the green(thick) line has been added in each path diagram to make evident the height of the lowest positive row count. Of course  each step ends with an updating of the row counts.


%
The final path diagram yields a path that is easily shown to be the desired pre-image. To obtain this path, we simply start with the leftmost red arrow, and at each step we proceed along the  arrow that starts at the rank reached by the previous arrow. Continue until all the arrows have been used.
 The reason why there always is a unique arrow that starts at each reached rank, is an immediate consequence of the $0$-row-count property
of the final path diagram.   The increasing property of the ranks of
  the starting  points of our   successive arrows, is now seen
to imply that the
initial path in Figure \ref{fig:InitialRank} is the sweep map image of this final path.
This manner of drawing  rational Dyck paths makes many needed properties more evident than the traditional manner and therefore also
easier to prove. As a case in point, we give a simple visual way of establishing the following nontrivial result (see, e.g., \cite{sweepmap}).

\begin{lem}\label{l-sweep-image-Dyck}
The sweep image of an
$(m,n)$-Dyck path is an $(m,n)$-Dyck path.
\end{lem}
\begin{proof}
\begin{figure}[!ht]
\centering{
\mbox{\includegraphics[height=2.8 in]{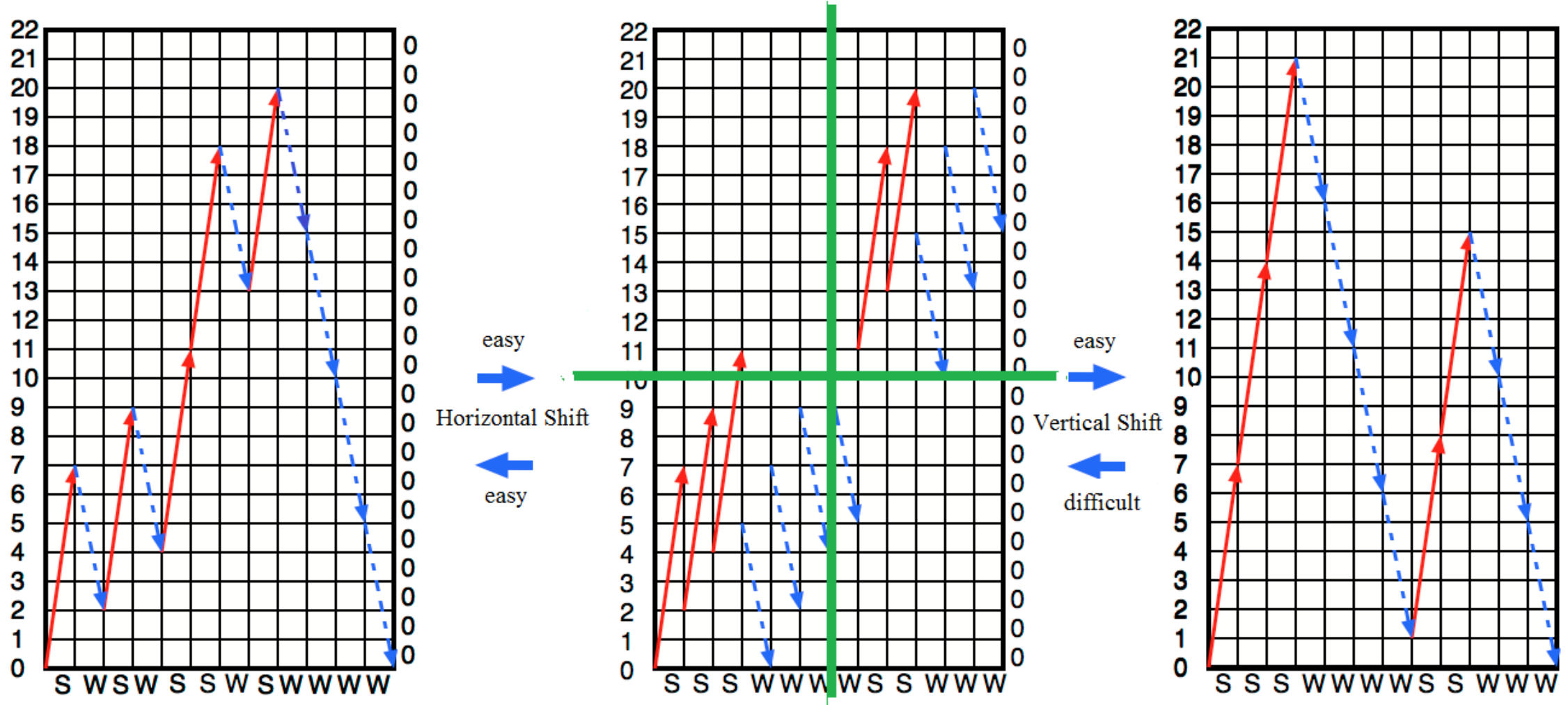}}
\caption{A $(7,5)$-Dyck path on the left; by horizontal shifts of the arrows we obtain the middle path diagram whose starting ranks are increasing; then by vertical shifts of the arrows to obtain a  $(7,5)$-Dyck  path.}
\label{fig:SweepMapDecompose}
}
\end{figure}

On the left of Figure \ref{fig:SweepMapDecompose} we have the final path yielded by our algorithm. To obtain the path diagram in the middle we simply rearrange the arrows (by horizontal shifts)
so that their starting ranks are increasing.  The path on the right is obtained
by vertically shifting the successive arrows so that they concatenate to a path. To  prove that the resulting path is a $(7,5)$-Dyck path, it follows, by the co-primality of the pair $(7,5)$, that   we need only show that the successive partial sums of the number of segments of these arrows are all non-negative.
This  is a consequence of the $0$-row-count property.  In fact, for  example,
 let us prove that the sum of the segments to the left of the vertical  green line $v$ is positive.

To this end, let $A$ be the arrow that starts on $v$ and $\ell$ be its starting rank. Let $h$ be the horizontal green (thick) line at level $\ell$.
Denote by  $ L$   the region  below   $h$, and let $L_1$, $L_2$ be the left and right portions of $L$ split by  $v$. Let us also denote by    $|rL_1|$,  $|rL_2|$,
the red arrow segment counts in the corresponding regions and
by    $|bL_1|$,  $|bL_2|$ the corresponding blue segment counts.
This given, since  red segments  contribute a  $1$ and a blue segment contributes   $-1$ to the final count, it follows that
$$
i)\,\,\,\,  |rL_1|+|rL_2| = |bL_1|+|bL_2|,
\,\,\,\,\,\,\,\,\,\,\,\,\,\,\,\,\,\,\,\,\,\,\,\,
ii)\,\,\,\,  |rL_2|= 0.
$$
In fact, $i)$ is due to the $0$-row-count property and $ii)$ is simply due to the  fact that  all red arrows to the right  of $v$ must start above $h$. Thus we must  have
$$
|rL_1| -  |bL_1| =|bL_2| \ge 0.
$$
This implies that the sum of the arrows to the left of $v$ must be $\ge 0$.
\end{proof}

A proof of the validity of our algorithm may be derived from the Thomas-Williams  result by letting  their modulus tend to infinity. However, our algorithm deserves a more direct and simple proof.

Such a proof will be given in the following pages. This proof will be based on
the validity of a simpler but less efficient algorithm. To distinguish the above algorithm from our later one. We will call them respectively the \texttt{StrongFindRank} and the \texttt{WeakFindRank} algorithms, or ``strong" and ``weak" algorithm for short.

The rest of the paper is organized as follows. Section 2 devotes to the proof of the \texttt{WeakFindRank} algorithms. It also includes all the necessary concepts and
concludes with Theorem \ref{t-important}, which asserts the invertibility of the rational sweep map. Theorem \ref{t-tight} is the main result of Section 3. It allows us to analyze
the complexity of both the ``strong" and the ``weak" algorithms. It is also used in Section 4, where we show the validity of the ``strong" algorithm. Finally, we discuss the difference between the Thomas-Williams algorithm and our algorithm in Section 5. We also talk about some future plans.

\section{The Proof}
\subsection{Balanced path diagrams}

A path diagram $T$ consists of an ordered  set of $n$  red arrows and $m$   blue arrows, placed on a $(m+n)\times N$ lattice rectangle. Where $N$ is a large positive integer to be specified.   See Figure \ref{fig:InitialRankWeak}.
\addpic[.30\textwidth]{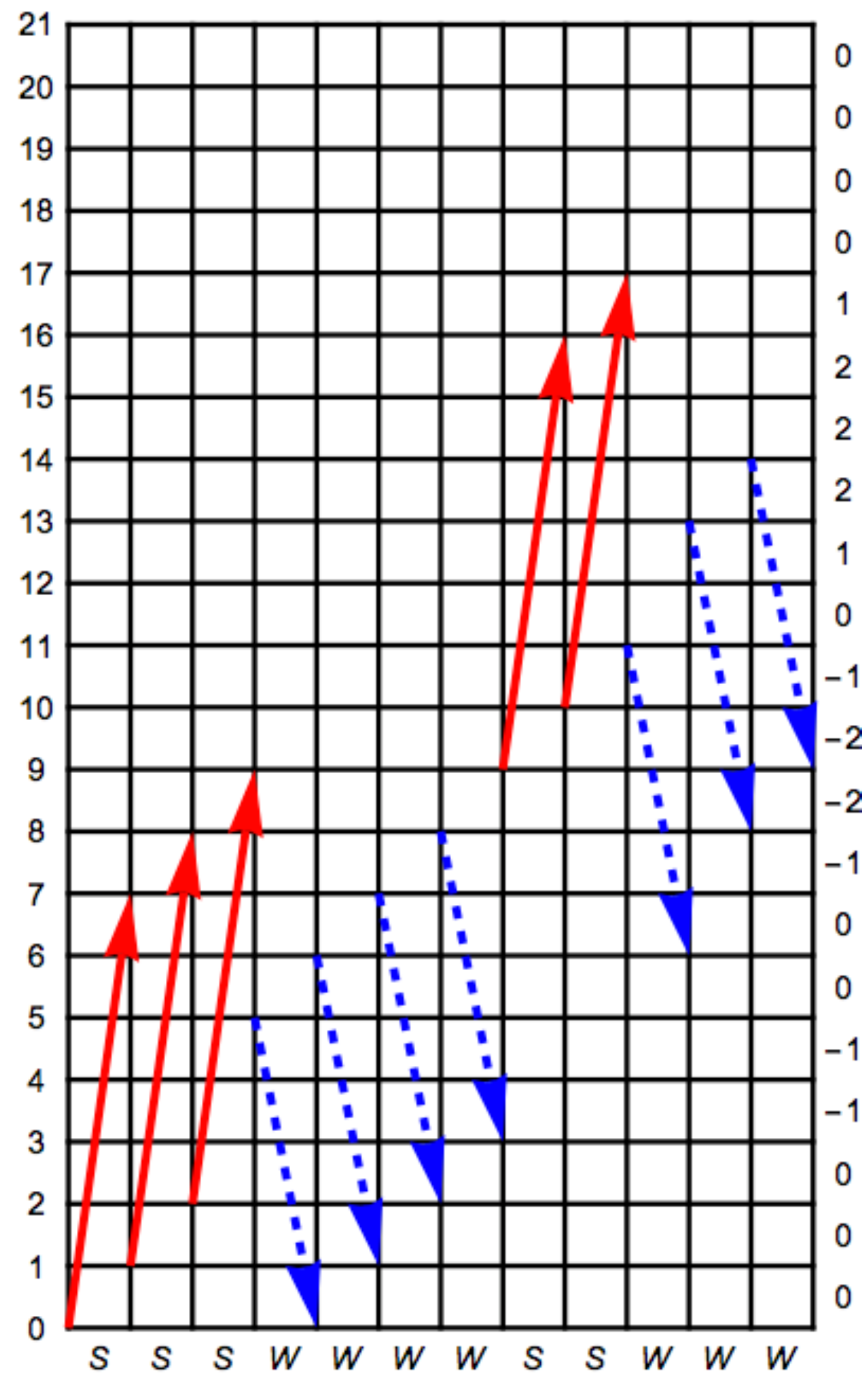}{ A path diagram for $m=5,n=7$ and $N=21$. \label{fig:InitialRankWeak}}

A red arrow is the  up  vector $(1,m)$ and a blue arrow is the down vector $(1,-n)$. The rows of lattice cells will be simply referred to as \emph{rows} and the horizontal lattice lines  will be simply referred to as \emph{lines}. On the left of each line we have placed its $y$ coordinate which we will simply refer to as its \emph{level}. The level of the starting point of an arrow is called its \emph{starting rank}, and similarly its \emph{end rank} is the level of its end point.
It will be convenient to call \emph{row $i$} the row of lattice cells delimited by the lines at levels $i$ and $i+1$.

Given a word $\Sigma$ with $n$ letters $S$ and $m$ letters $W$, and a sequence of $n+m$ nonnegative  ranks $R=(r_1,\dots, r_{m+n})$,
the path diagram $T(\Sigma,R)$ is obtained by placing the letters of $\Sigma$ at the bottom of the lattice columns and drawing in the  $i^{th}$ column an arrow with starting rank $r_i$ and $red$ (solid)
if the $i^{th}$ letter of $\Sigma$ is $\Sigma_i=S$ or $blue$ (dashed)
if $\Sigma_i=W$.
See Figure \ref{fig:InitialRankWeak},
where $\Sigma=SSSWWWWSSWWW$ and
$R=(0,1,2,5,6,7,8,9,10,11,12,13)$. Notice that each lattice
cell may contain a segment of a red arrow or  a segment of a blue arrow or no segment at all. The red segment count of row $j$ will be denoted $c^r(j)$ and the blue segment count
is denoted $c^b(j)$. We will set $c(j)=c^r(j)-c^b(j)$ and refer  to it as the \emph{count of row $j$}. In the above display on the right of each row we have attached its row count.
The following observation will be crucial in our development.
\begin{lem}\label{l-DifferenceLevelCount}
Let $T(\Sigma,R)$ be any path diagram. It holds for every integer $j\ge1$
that
\begin{align}
  c(j)-c(j-1)= \# \{\text{arrows starting at level } j\} -\#\{\text{arrows ending at level } j\}.
\end{align}
\end{lem}

\vspace{-4mm}
\addpic[.48\textwidth]{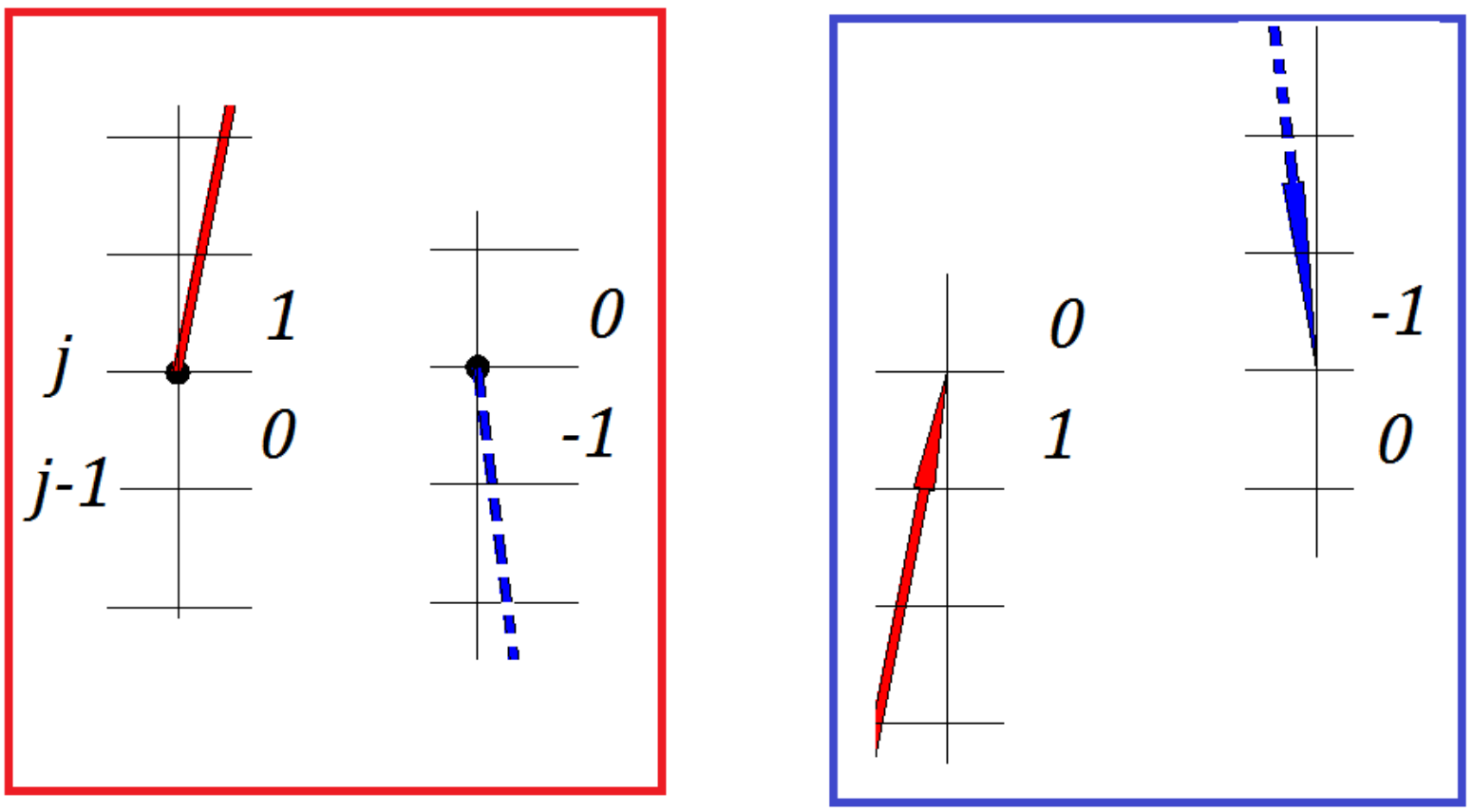}{The difference $c(j)-c(j-1)$ is $1$ in the left two cases,  is $-1$
in the right two cases, and is $0$ in the previous cases. \label{fig:FourCases}}

\begin{proof}
Let us investigate the contribution to the difference $c(j)-c(j-1)$ from a single arrow $A$.
The contribution is $0$ if i) $A$ has no segments in rows $j$ and $j-1$,
ii) $A$ has segments in row $j$ and $j-1$. In both cases,
it is clear that $A$ cannot start nor end at level $j$. Thus the remaining cases are as listed in Figure
\ref{fig:FourCases}.
\end{proof}

It will be convenient to say that a path diagram $T(\Sigma,R)$ is \emph{balanced}  if all
its  row counts are   equal to $0$. The word $\Sigma$ is said to be the \emph{$(S,W)$-word} of
a Dyck path $D$ in $\cal D_{m,n}$, if it is obtained by placing an $S$ when $D$
takes a South end (hence a North step) and a $W$ when $D$
takes a West end (hence an East step).

\begin{theo}\label{t-characterization}
Let $\Sigma$ be the $(S,W)$-word of $D\in \mathcal{ D}_{m,n}$, and
 let $R=(r_1,r_2,\ldots ,r_{m+n})$, with $r_1=0$, be a weakly   increasing sequence of integers.
Then  $R$ is a rearrangement of the  rank sequence of  a pre-image $\overline{D}$ of $D$  under the sweep map, if and only if the path diagram $T(\Sigma,R)$ is balanced and the
rank sequence $R$ is strictly increasing.
\end{theo}

\begin{proof}
Suppose that $\overline{D}$ is a pre-image of $D$. This given, let
$T(\overline{\Sigma},\overline{R})$ with $\overline{\Sigma}$ the $(S,W)$-word of $\overline{D}$,  
$\overline{R}$ the rank sequence of $\overline{D}$,
and  height $N$ chosen to be a number greater than  $nm+\max(\overline R)$. It is clear  that the arrows of
$T(\overline{\Sigma},\overline{R})$ can be depicted by starting
at level $0$ and drawing a red arrow $(1,m)$ every time $\overline{D}$ takes a South end and a blue arrow $(1,-n)$  every time $\overline{D}$ takes a West end, with each arrow starting where the previous arrow ended.  It is obvious that the row counts of $T(\overline{\Sigma},\overline{R})$ are all $0$ since, in  each row every red segment is followed by a blue segment. Now let $T(\Sigma,R)$ be the path diagram of same height $N$ with $\Sigma$ the $(S,W)$-word of $D$ and $R$ its rank sequence. By definition of  the Sweep map, the rank sequence  $R$ is obtained
by permuting  in  increasing order the  components of  $\overline{R} $. Since the co-primality of $(m,n)$ assures that $\overline R$ has distinct components, $R$ is necessarily a nonnegative  increasing sequence. Likewise, the word $\Sigma$ is obtained by rearranging the letters of $\overline {\Sigma}$ by the same permutation.
Thus we may say that  the same permutation can be used to change   $T(\overline{\Sigma},\overline{R})$ into $T( \Sigma , R )$. Since this operation only permutes segments  within  each row, it follows that all the row counts of $T( \Sigma , R )$ must also be  $0$. This proves the necessity.

For the sufficiency, suppose  that the path diagram $T(\Sigma,R)$ is balanced, with $D$ the Dyck path whose word is $\Sigma$  and $R$ a weakly increasing sequence.
Then by lemma 1 it follows that  for every level $j$, either i) no arrow  starts or ends at this level,  or ii) if  $k>0$  arrows  end   (start) at this level then exactly $k$ arrows    start (end) at this level.
This given, we will construct a Dyck path $\overline{D}$ by the following algorithm.
Starting at level $0$ we follow the first arrow, which we know is necessarily red
and starts at level $0$. This arrow ends at level $m$.
Since there is at least one arrow that starts at this level follow the very next arrow that does.  Proceeding recursively thereafter,
every time we reach a level, we follow the very  next arrow that starts at that level. This process stops when we are back at level $0$, and we must since in $\Sigma$ there are $n$ $S$ and $m$ $W$. Let $\overline D$ be the resulting path. Using the  colors of the successive arrows of $\overline D$ gives us the $\overline{\Sigma}$ word of $\overline{D}$. Now notice that  $\overline D$ must be a path in ${\cal D}_{m,n}$ since all its starting ranks are nonnegative due to the weakly increasing property of $R$  and therefore they must necessarily be distinct by the co-primality of $(m,n)$. In particular, if $\overline R$ denotes the sequence of starting ranks of $\overline D$ we are also forced to conclude that its components are distinct. Since the components of $\overline R$ are only a rearrangement of the components of $R$ we deduce that $R$ must have been
strictly increasing to start with. This implies that $D$ must be a Sweep map image of $\overline D$ since the  successive letters of $\Sigma$ can  be obtained by rearranging the letters of $\overline {\Sigma}$ by the same permutation that rearranges $\overline R$ to $R$. This completes the proof of sufficiency.
\end{proof}

\def \ovvR {\overline R}
\def \ovvr {\overline r}

This given,  we can easily see that the validity of our ``strong" algorithm hinges on establishing that it produces a balanced path diagram after a finite number of steps. Theorem 2 allows us to relax the strictly increasing
requirements on the rank sequences  of the successive  path diagrams produced
by the algorithm. The \texttt{WeakFindRank} algorithm, defined below, has precisely that property. This results  in a simpler proof of the termination property of both algorithms.

\subsection{Algorithm \texttt{WeakFindRank} and the Justification}\

\noindent
Algorithm \texttt{WeakFindRank}

\noindent
\textbf{Input:} A path diagram $T(\Sigma,R^{(0)})$ with  $\Sigma$ the word of a Dyck path $D\in \cal D_{m,n}$, a weakly increasing rank sequence $R^{(0)} =(r_1^{(0)},r_2^{(0)},\ldots ,r_{m+n}^{(0)})$.

\noindent
\textbf{Output:} A balanced path diagram $T(\Sigma,\ovvR)$.

It will be convenient to keep the common height equal to $N$ for all the successive  path diagrams constructed by the algorithm, where $N=U+2mn$, with $U=\max(R^{(0)})+m+1$.
\begin{enumerate}
\item[Step 1] { Starting with $T(\Sigma,R^{(0)})$ repeat the following step
until the resulting path diagram is balanced. }
\vskip .1in

\item[Step 2]
{ In $T(\Sigma,R^{(s)})$,  with  $R^{(s)} =(r_1^{(s)},r_2^{(s)},\ldots ,r_{m+n}^{(s)}),$  find the lowest row $j$ with $c(j)>0$ and
find the rightmost  arrow that starts at level $j$. Suppose that arrow
starts at $(i,j)$. Move up the arrow one level to construct the path diagram  $T(\Sigma,R^{(s+1)})$ with $ r_{i}^{(s+1)}= r_{i}^{(s)}+1$
and  $ r_{i'}^{(s+1)}= r_{i'}^{(s)}$ for all $i'\neq i$. If all the row counts are $\le 0$ then stop the algorithm, since all row counts
must  necessarily vanish.}
\end{enumerate}

\addpic[.45\textwidth]{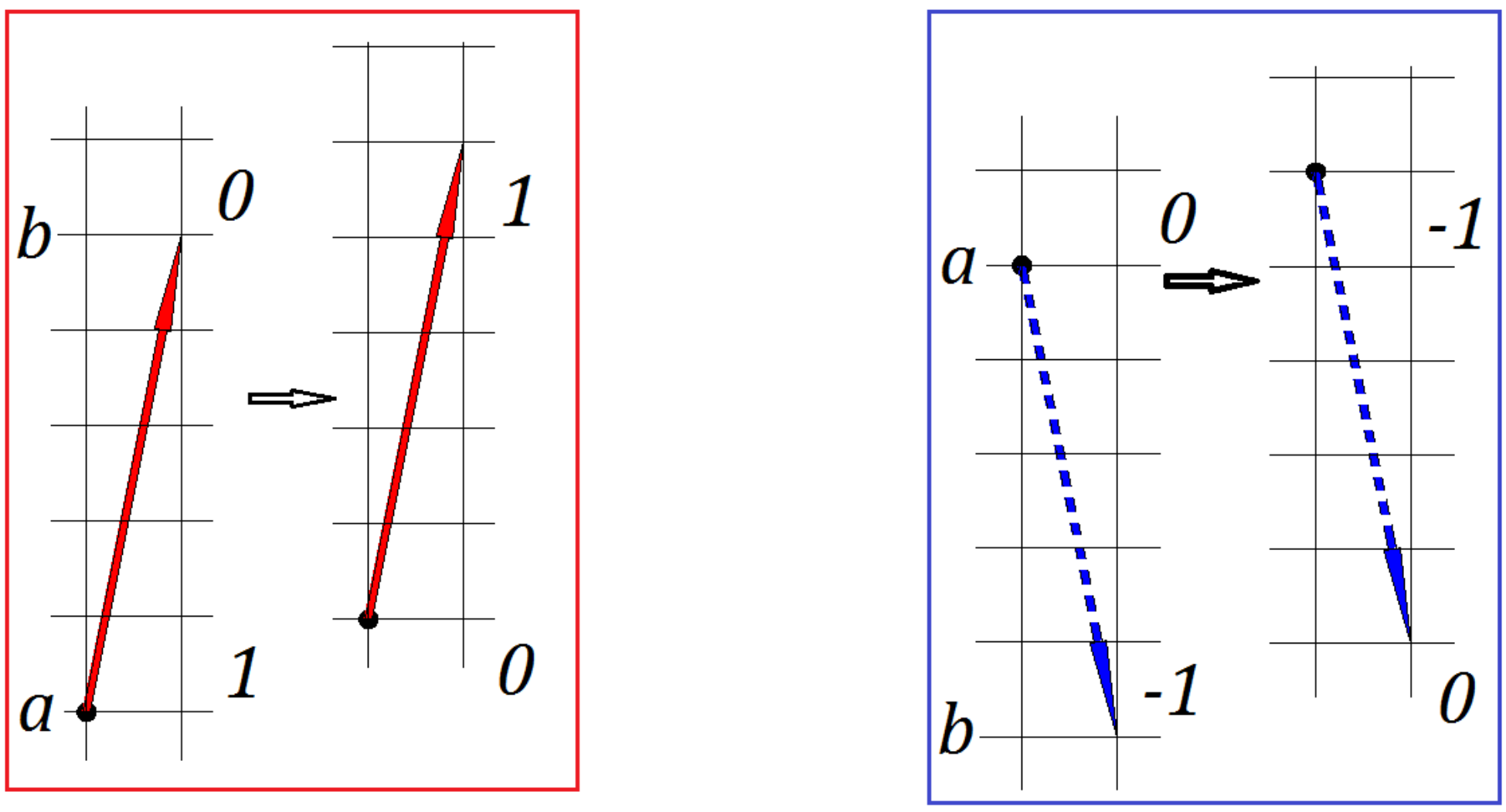}{Shifting up one unit an arrow from  level $a$ to level $b$ will decrease $c(a)$ by 1 and increase $c(b)$ by 1. \label{fig:ShiftArrow}}

Figure \ref{fig:ShiftArrow} shows that we are weakly reducing the number of rows with positive row counts in Step 2.
It also makes the following key observation evident.
\begin{lem}\label{l-keep-nonnegative}
If  at some point  $c(k)$ becomes $\ge 0$ then for ever after it will never become negative. In particular, since $c(k)= 0$ with $k>U$ for the initial path diagram $T(\Sigma,R^{(0)})$ we will have $c(k)\ge 0$ when $k>U$ for all successive path diagrams produced by the algorithm.
\end{lem}
\begin{proof}
The lemma holds true
because we only decrease a row count when it is positive.
\end{proof}

We need some basic properties to justify the algorithm.
\begin{lem}\label{l-basic-property}
We have the following basic properties.
\begin{enumerate}
\setlength\itemsep{2mm}
\item[$(i)$] { If row $j$ is the lowest with $c(j)>0$ then there is an arrow   that starts at level  $j$. In this situation, we say that we are \emph{working} with row $j$.}

\item[$(ii)$] { The successive rank sequences are  always weakly  increasing.}

\item[$(iii)$]  { If $T(\Sigma,R)$ has no positive row counts, then it is balanced. Consequently, if the algorithm terminates, the last path diagram is balanced.}
\end{enumerate}
\end{lem}

\begin{proof}\

\begin{enumerate}\setlength\itemsep{2mm}

\item[$(i)$]    By the choice of $j$, we have $c(j)>0$ and $c(j-1)\le 0$. Thus $c(j)-c(j-1)>0$, which by Lemma \ref{l-DifferenceLevelCount}, shows that there is at least one arrow starting at rank $j$.

\item[$(ii)$]   Our choice of $i$ in step (2) assures that the next rank
sequence remains weakly increasing.

\item[$(iii)$]  Since each of our path diagrams  $T(\Sigma,R)$ has $n$ red arrows
of length $m$ and $m$ blue arrows of length $n$, the total sum of row counts of any $T(\Sigma,R)$ has to be $0$. Thus if $T(\Sigma,R)$ has no positive row counts, then it must have no negative row counts either, and is hence balanced.
 \end{enumerate}
\vspace{-8mm}
 \end{proof}

\begin{proof}[Justification of Algorithm \texttt{WeakFindRank}]
By Lemma \ref{l-basic-property}, we only need to show that the algorithm terminates.
To prove this we need the following auxiliary result.

\begin{lem}\label{l-redcount}
Suppose we are working with row $k$, that is  $c(k)>0$ and $c(i)\le 0$ for all $i<k$. If row $\ell $ has no segments for some $\ell <k$, then the current path diagram $T(\Sigma,R')$ has no segments below row $\ell$.
\end{lem}
\begin{proof}\hspace{-2cm}\vspace{-1cm}\addpic[.38\textwidth]{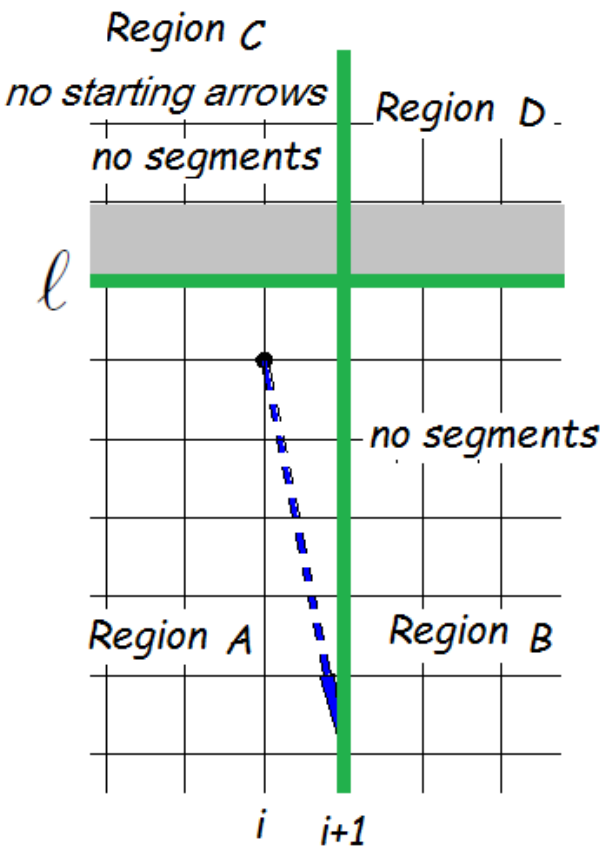}{When row $\ell$ has no segments, there will be no segments in regions $B$ and  $C$.\label{fig:emptyrow}}
\vspace{.1cm} \hspace{5mm}
Suppose to the contrary that $T(\Sigma,R')$ has  a segment below row $\ell$, then
let $V$ be the right most arrow that contains such a segment and say that it starts
at column $i$. Since row $\ell$ has no segments, the   starting rank of $V$ must be $\le \ell$. This implies that $i+1<m+n$ since the arrow that starts at level $k$ must be to the right of $V$ (by the increasing property of $R'$).  This given,  the current path diagram could look  like in
Figure \ref{fig:emptyrow}, where the two green (thick) lines divide the plane into 4 regions, as labelled in the display.

The weakly increasing property of $R'$  forces no starting ranks in  $C$, therefore there are no segments there. By the choice of $V$ there cannot be any segments   in $B$. Thus the  (gray) empty row $\ell$  forces no segments within both
$B$ and $C$.

Now notice that since $\Sigma$ is the word of a path $D\in \mathcal D_{m,n}$ the number of red segments to the left of column $i+1$ minus the number of blue segments to the left of that column must result in a number $s>0$.
 However, since $c(j)\le 0$ for all $j\le \ell$ it follows that $c(0)+\cdots +c(\ell-1)\le 0$. But since regions $B$ and $C$ have no segments it also follows that $s=c(0)+\cdots +c(\ell-1)\le 0$, a contradiction.
\end{proof}

 \vskip -.1in
Next observe that since each step of the algorithm increases one of the ranks by one unit, after $M$ steps we will have  $|R^{(M)}-R^{(0)}|=\sum_{i=1}^{m+n}r_i^{(M)}-\sum_{i=1}^{m+n}r_i^{(0)} =M$. This given, if the algorithm iterates Step 2 forever, then the maximum rank  will eventually exceed any given integer. In particular, we will end up working  with row $k$ with $k$ so  large that $k-U$ exceeds the total number $mn$ of red segments. At that point we will have $c(k)>0$ and  $c(j)=0$ for all
the $k-U$ values $j=U,U+1,\dots, k-1$. The reason for this is that we must have $c(j)\le 0$ for all $0\le j < k$ and by Lemma \ref{l-keep-nonnegative} we must also have
$c(j)\ge 0$ for $j\ge U$. Now, by the  pigeon hole principle, there must also be
some $U\le \ell <k$ for which $c^r(\ell )=0$. But then it follows that $c^b(\ell)=c^r(\ell)-c(\ell)=0$, too. That means that row  $\ell$ contains no segments.  Then   Lemma \ref{l-redcount} yields that there cannot be any segments  below row $\ell$ either. This implies that the  total row count is
$
\sum_{j\ge 0} c(j) = \sum_{j\ge U} c(j) \ge c(k)>0
$,
a contradiction.
\end{proof}

Thus the \texttt{WeakFindRank} algorithm  terminates and we can draw the following important conclusion.

\begin{theo}\label{t-onto}
Given any $(m,n)$-Dyck path $D$ with $(S,W)$-word $\Sigma$ and any initial weakly increasing rank sequence $R^{(0)}=(r_1^{(0)},r_2^{(0)},\cdots ,r_{m+n}^{(0)})$,
let $T(\Sigma,\oR)$ be the balanced path diagram produced by the \texttt{WeakFindRank} algorithm. Then the rank sequence \\ $\oR=(\orr_1,\orr_2,\cdots ,\orr_{m+n})$ will
be strictly increasing. Moreover, the sequence
$$
\widetilde R= (0,\orr_2-\orr_1,\orr_3-\orr_1,\ldots,\orr_{m+n}-\orr_1)
$$
is none other than the increasing rearrangement of the rank sequence of
a pre-image $\overline D$ of $D$ under the sweep map.
\end{theo}

\vskip -.5in
\begin{proof}
Clearly, the path diagram $T(\Sigma,\widetilde R)$
of height $N={\widetilde r}_{m+n}+m+1$ will also be balanced. Thus, by Theorem 2,   $\widetilde R$ must be the increasing
 rearrangement of the  rank sequence of  a pre-image $\overline{D}$ of
$D$. In particular not only $\widetilde R$ but also $\oR$ itself must be strictly increasing.
\end{proof}

 This result has the following important corollary.
  \begin{theo}\label{t-important}
  For any co-prime pair $(m,n)$ the rational $(m,n)$-sweep map is invertible.
 \end{theo}
 \begin{proof}
Lemma \ref{l-sweep-image-Dyck} shows that the rational $(m,n)$-sweep map is into. Theorem \ref{t-onto} gives a proof (independent of the Thomas-Wiliams proof)
 that it is onto. Since the collection of  $(m,n)$-Dyck paths is finite,
 the sweep map must be bijective.
 \end{proof}

Figure \ref{fig:fourpics} depicts the entire history of the  \texttt{WeakFindRank} algorithm applied to a $(7,5)$-Dyck path $D$. Both $D$ (on the left) and its pre-image
$\overline D$ (on the right) are depicted below.
\begin{figure}[!ht]
\centering{
\mbox{\includegraphics[height=5.5in]{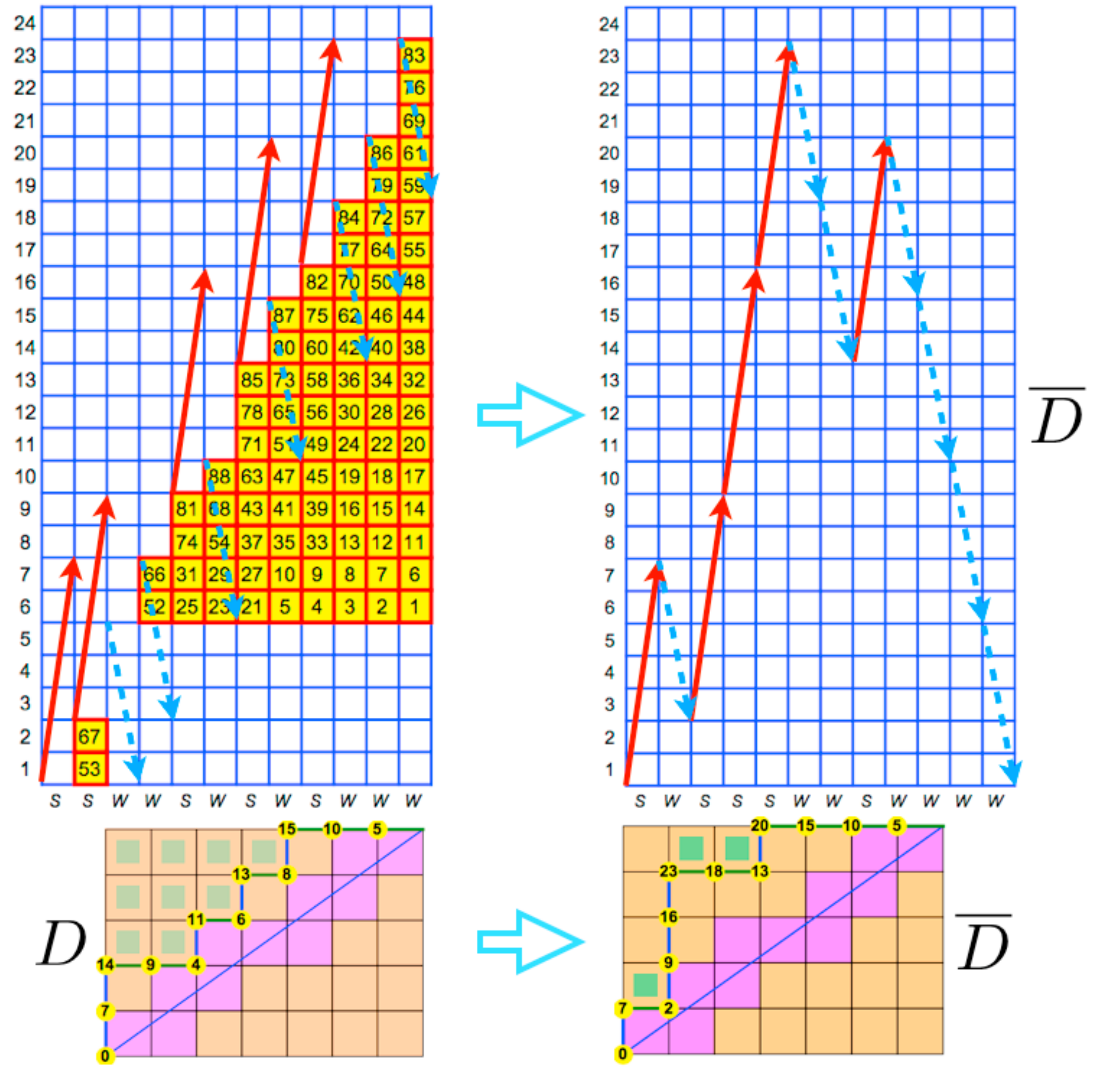}}
}
\caption{The top left picture is the balanced path diagram produced by the \texttt{WeakFindRank} algorithm. The top right picture
is obtained by Theorem \ref{t-characterization}. 
 \label{fig:fourpics}  }
\end{figure}

\vskip -.01in

The initial rank sequence that was used here is $R^{(0)}=(0,0,5,5,5,5,5,5,5,5,5)$. A boxed lattice square in column $i$ with
an integer $k$  inside  indicates that  arrow $A_i$ was processed at
 the $k^{th}$  step of the algorithm. As a result $A_i$ was lifted  from the level  of the bottom of the square to its top level. For instance the square with $63$ inside indicates that the red  arrow $A_7$ was lifted at the $63^{rd}$ step of the algorithm from starting at level $9$ to starting at level $10$. We also see that the last time that the arrow $A_8$ reached its final starting level at step $87$. The successive final starting levels of arrows $A_1,A_2,\ldots ,A_{12}$
 give  the increasing rearrangement of the ranks of the path $\overline D$.
 Notice, arrows $A_1$ and $A_3$ where never lifted. 

\def \TR {\widetilde R}
\def \tr {\widetilde r}
\def \oR {\overline R}
\def \orr {\overline r}

\section{Tightness of Algorithm \texttt{WeakFindRank} }
Following the notations in Theorem \ref{t-onto}, the number of steps needed for Algorithm \texttt{WeakFindRank} is
$|\overline R|-|R^{(0)}| = |\TR|-|R^{(0)}| + (m+n) \bar r_1.$ We will show that a specific starting path diagram can be chosen so that $\bar r_1 =0$.

For  two rank sequences $R=(r_1,r_2,\ldots ,r_{n+m})$ and
$R'=(r_1',r_2',\ldots ,r_{n+m}')$ let us write $R\preceq R'$ if and only if
we have $r_i\le r_i'$ for all $1\le i\le m+n$, if $r_i< r_i'$ for at least one $i$
we will write $R\prec R'$. The \emph{distance} of $R$ from $R'$,
will be expressed by the integer
 $$
| R'-R| \, =\, \sum_{i=1}^{m+n} (r_i'-r_i)=\, \sum_{i=1}^{m+n} r_i'- \sum_{i=1}^{m+n} r_i.
$$

Given the $(S,W)$-word $\Sigma$ of an $(m,n)$-Dyck path $D$, we will
call  the initial starting sequence $R^{(0)}$  \emph{canonical for $\Sigma$} if it is obtained by replacing the first string of $S$ in $\Sigma$ by $0$'s and all the remaining letters by $m$, and call the balanced path diagram $T(\Sigma,\TR)$
yielded by Theorem \ref{t-onto} \emph{canonical for $\Sigma$}. Clearly $R^{(0)} \preceq \TR$.
This given, we can prove the following remarkable result.

\begin{theo}\label{t-tight}
Let $\Sigma$ be the $(S,W)$-word of a Dyck path $D\in \mathcal{D}_{m,n}$.
If $R^{(0)}$ and $T(\Sigma,\TR)$ are  canonical for $\Sigma$ and $R$ is any  increasing sequence which satisfies the inequalities
$$
R^{(0)} \preceq R\preceq \TR
$$
then the  \texttt{WeakFindRank} algorithm with starting path diagram $T(\Sigma,R)$
will have as output the rank sequence $\TR$.
\end{theo}

\begin{proof}
Notice if $|\TR-R|=0$ there is nothing to prove. Thus we will proceed
by induction on the distance of $R$ from $\TR$.
Now assume the theorem holds for $|\TR-R|=K$.
We need to show that it also holds for $|\TR-R|=K+1$.
Suppose one application of step (2) on $R$ gives $R'$.
We need to show that $R'\preceq \TR $. This done since $R$ and $R'$ only differ from one unit  we will have $|\TR-R'|=K$ and then the inductive hypothesis would complete the proof.

Thus assume if possible that this step (2) cannot be carried out because it requires increasing by one unit an $r_i=\tr_i$ .
Suppose further that under this  step (2) the level $k$ was the lowest with $c(k)>0$ and thus the arrow $A_i$ was the right most that started at level $k$. In particular this means that $r_i=\tr_i=k$. Since $|  \TR-R|\ge 1$, there is at least one $i'$ such that $r_{i'}<\tr_{i'}$. If $r_{i'}=k'$ let $i'$ be the right most
with $r_{i'}=k'$. Define $R''$ to be the rank sequence obtained by replacing  $r_{i'}$ by
$r_{i'}+1$ in $R$. The row count $c(k')$ is decreased by $1$ and another row count is increased by $1$, so that $c(k)$ in $R''$ is still positive.
Since $|\TR-R''|=K$ the induction hypothesis assures that
the  \texttt{WeakFindRank} algorithm will return $\TR$. But then in carrying this out, we have to work on row $k$, sooner or later, to decrease the positive row count $c(k)$. But there is no way the arrow $A_i$ can stop being the right most starting at level $k$, since arrows to the right of $A_i$ start at a higher level than $A_i$  and  are only moving upwards.   Thus the fact that   the  \texttt{WeakFindRank} algorithm outputs  $\TR$ contradicts that fact that the application of step (2) to $R$ cannot be carried out.

Thus we will be able to lift $A_i$ one level up as needed and obtain the sequence $R'$ obtained by replacing $r_i$ by $r_i+1$ in $R$.
But now  we will have $|\TR-R'|$ =K$ $ with  $R'\prec \TR$
and the inductive hypothesis will assure us that the
\texttt{WeakFindRank} algorithm starting from $R$ will return $\TR$
as asserted.
\end{proof}

\addpic[.295\linewidth]{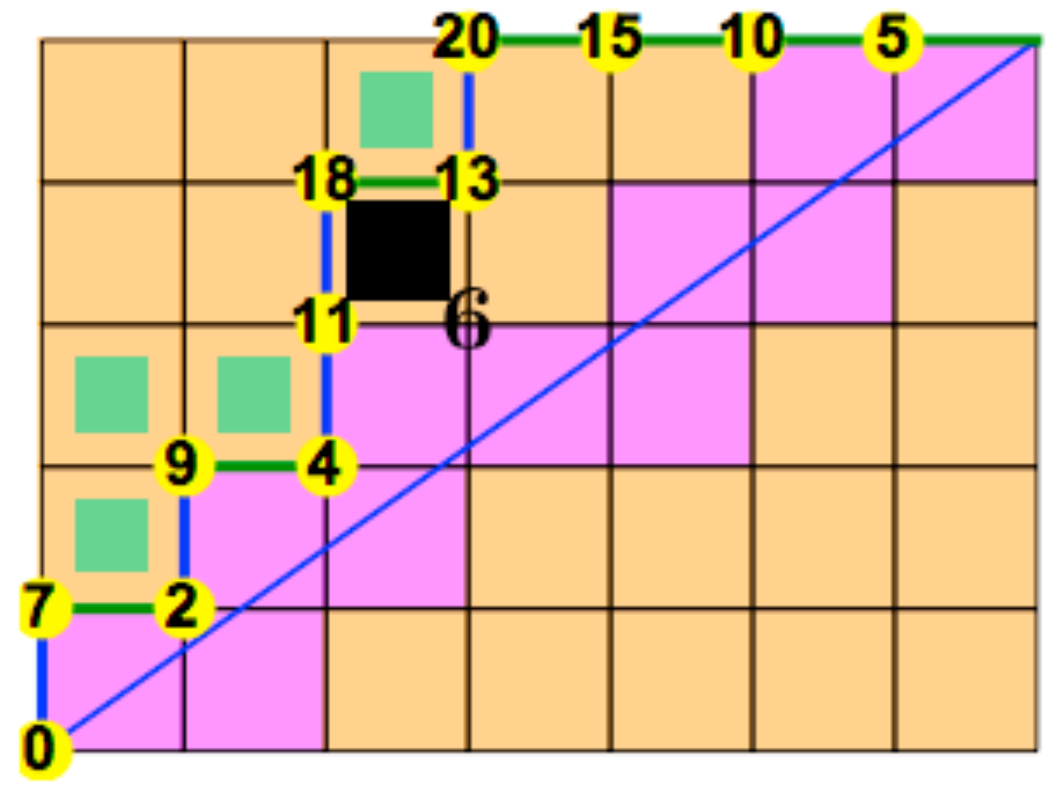}{A \\
$(7,5)$-Dyck path with area $4$. Removing the black cell changes the rank $18$ to $18-7-5=6$.
\label{fig:RemoveASquare} }

It is clear now that the complexity of the \texttt{WeakFindRank} algorithm is $O(|\TR|)$.
Recall that
reordering $\TR$ gives the rank set $\{r_1,r_2,\dots, r_{m+n}\}$ of $\overline D$. It is known and easy to show that
$$ \textrm{area} (\overline D)= \frac{1}{m+n}\left( \sum_{i=1}^{m+n} r_i -\binom{m+n}{2}  \right),\qquad\qquad\qquad\qquad\qquad $$
where $\textrm{area} (\overline D)$ is the number of lattice cells between $\overline D$ and the diagonal.
Indeed, from Figure \ref{fig:RemoveASquare},
it should be evident that reducing the area by $1$ corresponds to reducing the sum $r_1+\cdots +r_{m+n}$ by $m+n$.

It follows that
\begin{align}
  \label{e-TR-area}
 |\TR| = (m+n)\textrm{area} (\overline D)+\binom{m+n}{2} =O((m+n)\textrm{area} (\overline D)) .\qquad\qquad\qquad\qquad\qquad
\end{align}

Theorem \ref{t-tight} together with \eqref{e-TR-area} gives the following result.
\begin{cor}
Given any $(m,n)$-Dyck path $D$, its pre-image $\overline D$ can be produced in $O((m+n) \area( \overline D))$ running time.
\end{cor}
\begin{proof}
Let $\Sigma$ be the $(S,W)$-word of $D$. We first construct the path diagram $T=(\Sigma, R^{(0)})$ with  $R^{(0)}$ being canonical for $\Sigma$ and compute
the row counts of the path diagram.
Next we use the \texttt{WeakFindRank} algorithm to update $T$ until we get the balance path diagram $(\Sigma, \TR)$ by Theorem \ref{t-tight}. Finally we use Theorem
\ref{t-onto} to find the pre-image $\overline D$.

Iteration only appears in the middle part, where the \texttt{WeakFindRank} algorithm  performs $|\TR|-  |R^{(0)}|$ times of Step 2. In each Step 2, we search for the lowest positive row count $c(j)$, then search for the rightmost arrow $r_i$ that is equal to $j$, and finally update $r_i$ by $r_i+1$ and the row counts at \emph{only} two rows (see Figure \ref{fig:ShiftArrow}). Therefore, the total running time is $O( |\TR|) $, and the corollary follows by \eqref{e-TR-area}.
\end{proof}

%

\def \HR {\widehat R}
\def \Hr {\widehat r}

\def\stco{\mathop{\texttt{sc}}}

\section{Validity of Algorithm \texttt{StrongFindRank} }
Let $R=(r_1\le r_2\le \cdots \le r_l)$ be a  sequence of nonnegative integers of length $l$. The \hbox{{\it strict cover $R'=\stco (R)=(r_1'<r_2'< \cdots < r_l')$ of $R$}} is recursively defined by
$r_1'=r_1$ and $r_i'=\min(r_{i},r_{i-1}'+1)$ for $i\ge 2$. It is the unique minimal strictly increasing sequence satisfying $R'\succeq R$. The following principle is straightforward.

\medskip
\centerline{\emph{ If $R\prec \oR$ with $R$ weakly increasing and $\oR$ strictly increasing, then $\stco (R) \preceq\oR
$}.}

A direct consequence is that $\HR^{(0)} = \stco( R^{(0)})  \preceq \TR$ if $R^{(0)}$ is canonical for $\Sigma$. This sequence is exactly the starting rank sequence
of our ``strong" algorithm (see Figure \ref{fig:InitialRank}). It will be good to  review our definitions before we proceed.

\noindent
Algorithm
\texttt{StrongFindRank}

\noindent
\textbf{Input:} A path diagram $T(\Sigma,\HR^{(0)})$ with $\Sigma$ the word of a Dyck path $D\in \cal D_{m,n}$, the nonnegative strictly increasing rank sequence $\HR^{(0)}$ as above.

\noindent
\textbf{Output:} The balanced path diagram $T(\Sigma, \HR)$.

It will be convenient to keep the common height equal to $N$ for all the successive  path diagrams constructed by the algorithm, where $N=U+2mn$, with $U=\max(R^{(0)})+m+1$.

\begin{enumerate} \setlength\itemsep{2mm}

\item[Step 1] {Starting with $T(\Sigma,\HR^{(0)})$ repeat the following step
until the resulting path diagram is balanced. }

\item[Step 2]
{In $T(\Sigma,\HR^{(s)})$,  with  $\HR^{(s)} =(\Hr_1^{(s)},\Hr_2^{(s)},\ldots ,\Hr_{m+n}^{(s)})$  find the lowest row $j$ with $c(j)>0$ and
find the unique  arrow that starts at level $j$. Suppose that arrow
starts at $(i,j)$. Define $R'$ to be the rank sequence obtained from
$\HR^{(s)}$ by the \hbox{replacement}
\hbox{$\Hr_i^{(s)}\longrightarrow\Hr_i^{(s)}+1$}, and set
$\HR^{(s+1)}=\stco(R')$.
Construct the path diagram  \hbox{$T(\Sigma,\HR^{(s+1)})$} and update the row counts.
If all the row counts of $T(\Sigma,\HR^{(s)})$ are $\le 0$ then stop the algorithm  and return $\HR^{(s)} $}, since all these row counts must  vanish.
\end{enumerate}

This given, the validity of the \texttt{StrongFindRank} algorithm is an immediate consequence of the following surprising result.

\begin{theo}\label{t-surprise}
Let $D\in \mathcal{D}_{m,n}$ with $(S,W)$-word $\Sigma$, and let the balanced path diagram
$(\Sigma,\TR)$ be canonical for $\Sigma$. Then
all the
successive rank sequences  $\HR^{(s)}$ produced by the \texttt{StrongFindRank} algorithm satisfy the inequality
\begin{align}
  \label{e-1}
\HR^{(s)}\preceq \TR
\end{align}
and since the successive rank sequences  satisfy the  inequalities
\begin{align}\label{e-2}
 \HR^{(0)}\prec \HR^{(1)}\prec  \HR^{(2)}\prec \HR^{(3)}\prec
\cdots \prec  \HR^{(s)},
\end{align}
there will necessarily come a step when $T(\Sigma, \HR^{(s)})=T(\Sigma,\TR)$. At that time the algorithm will stop and output $\TR$.
\end{theo}

\begin{proof}
The inequality \eqref{e-2} clearly holds since we always shift arrows upwards.

We prove the inequality \eqref{e-1} by induction on $s$. The basic fact that will play a crucial role is that the output $\TR$
is strictly increasing. See Theorem \ref{t-characterization}.

The case $s=0$ of \eqref{e-1} is obviously true since $\HR^{(0)}$ is the strict cover of
$R^{(0)} \preceq \TR$. Assume $\HR^{(s)}\preceq \TR$ and we need to show \eqref{e-1} holds true for $s+1$.
Now $\HR^{(s+1)}$ is the strict cover of $R'$, where  $R'$ is the auxiliary rank sequence  used by Step 2 of the \texttt{StrongFindRank} algorithm.
Since $R'$ is precisely the successor of $\HR^{(s)}$ by Step 2 of the  \texttt{WeakFindRank} algorithm, it will necessarily satisfy the inequality $R'\prec \TR$ by Theorem \ref{t-tight}. Our  principle then guarantees that we will also have
$$
\HR^{(s+1)}\prec \TR
$$
unless $\HR^{(s)}= \TR$ and the \texttt{StrongFindRank} algorithm terminates.
\end{proof}

\begin{rem}
This proof makes it evident that, to  construct  the pre-image of an $(m,n)$ Dyck path, the \texttt{StrongFindRank} algorithm will be more efficient than the \texttt{WeakFindRank} algorithm. This  is partly due to the fact that  the distances  $|\HR^{(s+1)}-\HR^{(s)}|$ do turn out bigger than one unit
most of the time, as we can see in the following display.

\end{rem}

\begin{figure}[!ht]
\centering{
\mbox{\includegraphics[height=4in]{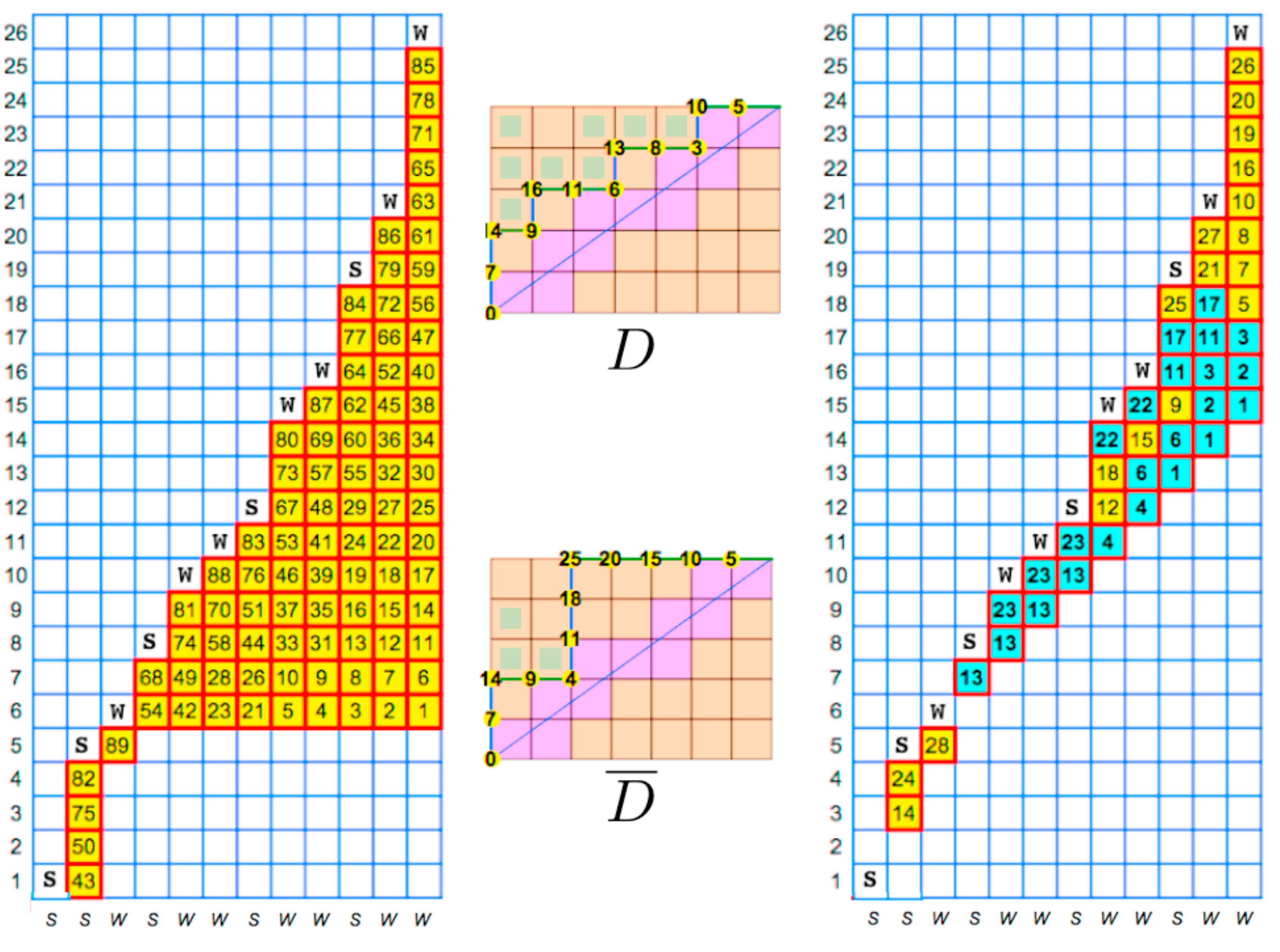}}
}
\caption{Comparison of the WeakFindRank algorithm and the StrongFindRank algorithm by an example.\label{fig:MAGO}}
\end{figure}

In the middle of Figure \ref{fig:MAGO} we have a Dyck path $D$, and below it, its pre-image $\overline D$. To recover $\overline D$ from $D$ we applied to $D$ the
\texttt{WeakFindRank} algorithm (on the left) and  the \texttt{StrongFindRank} algorithm  (on the right). The display shows that the ``weak" algorithm required about $3$ times more steps than the  ``strong"   algorithm.
The numbers in the Cyan squares reveal that, in several steps,
two or more arrows were lifted at the same time. On one step as many as $4$ arrows where lifted (step $13$).  The other step saving feature
of the ``strong" algorithm  is due to   starting from
the strict cover of the canonical starting sequence. This is evidenced
by the difference between the number of white cells below the colored ones on  the left and on the right diagrams.

\section{Discussion and Future Plans}
This work is done after the authors read \cite{Nathan} version 1, especially after the first named author talked with Nathan Williams. The concept ``balanced path diagram" is a translation of ``equitable partition" in \cite{Nathan}. The intermediate object ``increasing balanced path diagram" is what we missed in our early attempts:
The obvious $0$-row-count property of Dyck paths gives the necessary part of Theorem \ref{t-characterization}, but we never considered the $0$-row-count property to be sufficient until we read the paper \cite{Nathan}.

Once Theorem \ref{t-characterization} is established, inverting the sweep map is reduced to searching for the corresponding increasing balanced path diagram. Our algorithm is similar to the Thomas-Williams algorithm in the sense that
both algorithms proceed by picking an initial candidate and then repeat an identical updating process until terminates. In the rational Dyck path model, our updating process is natural and has more freedom than the Thomas-Williams algorithm. For instance, we can start with any weakly increasing rank sequence.

In an upcoming paper, we will extend the arguments in this paper to a more general class of sweep maps. Such sweep maps have been defined in \cite{sweepmap}. Though
the invertibility of these sweep maps can be deduced from Nathan's modular sweep map model \cite{Nathan},  they deserve direct proofs.

Even the rational sweep map needs further studied. The $(m,n)$-rational sweep map on $\mathcal D_{m,n}$ is known to take the \texttt{dinv} statistic to the \texttt{area} statistic.
This result is proved combinatorially by Gorsky and Mazin in \cite{Gorsky-Mazin}, but the proof is indirect. Our view of Dyck paths leads to  visual description of the \texttt{dinv} statistics and a simple proof of the \texttt{dinv} and \texttt{area} result. See \cite{dinv-area} for detailed information and references.

{\small \textbf{Acknowledgements:}

The first named author is grateful to Nathan Williams  for the time and effort that he spent to communicate his pioneering proof of invertibility of the general sweep map.

The first named author was supported by NFS grant DMS13--62160. The second named author was supported by National Natural Science Foundation of China (11171231).

\end{document}